\documentclass{amsart}
\usepackage{graphicx}
\theoremstyle{definition}
\newtheorem{Def}{Definition}
\theoremstyle{plain}
\newtheorem{Thm}[Def]{Theorem}
\newtheorem{Lem}[Def]{Lemma}
\begin{document}
\title[Bridge decompositions]{Bridge decompositions with distances at least two}
\author[K. Takao]{Kazuto Takao}
\address{Department of Mathematics, Graduate school of science, Osaka University.}
\email{u713544f@ecs.cmc.osaka-u.ac.jp}
\begin{abstract}
For $n$-bridge decompositions of links in $S^3$, we propose a practical method to ensure that the Hempel distance is at least two.
\end{abstract}
\maketitle

\section{Introduction}

Hempel distance is a measure of complexity originally defined for Heegaard splittings of $3$-manifolds \cite{hempel}.
The definition can be extended to bridge decompositions of links and it has been successfully applied to knot theory.
For example, extending Hartshorn's \cite{hartshorn} study for Heegaard splittings, Bachman-Schleimer \cite{bachman-schleimer} showed that the distance of a bridge decomposition of a knot bounds from below the genus of any essential surface in the knot exterior.
Extending Scharlemann-Tomova's \cite{scharlemann-tomova} for Heegaard splittings, Tomova \cite{tomova} showed that the distance of a bridge decomposition bounds from below the bridge number of the knot or the Heegaard genus of the knot exterior.

However, it is difficult to calculate the Hempel distance of a general Heegaard splitting or bridge decomposition.
While estimating it from above is a simple task in principle, it is a hard problem to estimate the distance from below.

For a Heegaard splitting, Casson-Gordon \cite{casson-gordon} introduced the rectangle condition to ensure that the distance is at least two.
Lee \cite{lee} gave a weak version of rectangle condition which guarantees the distance to be at least one.
Berge \cite{berge} gave a criterion for a genus two Heegaard splitting which guarantees the distance to be at least three.
Lustig-Moriah \cite{lustig-moriah} also gave a criterion to estimate the distance of a Heegaard splitting from below.

On the other hand, we could not find corresponding results for bridge decompositions in literature.
In this paper, we observe that a bridge decomposition of a link in $S^3$ can be described by a {\it bridge diagram}, and show that the {\it well-mixed condition} for a bridge diagram guarantees the distance to be at least two (see Section \ref{diagram} for definitions).
It may be regarded as a variation of the rectangle condition for Heegaard diagrams.

\begin{Thm}\label{main}
Suppose $(T_+,T_-;P)$ is an $n$-bridge decomposition of a link in $S^3$ for $n\geq 3$.
If a bridge diagram of $(T_+,T_-;P)$ satisfies the well-mixed condition, the Hempel distance $d(T_+,T_-)$ is at least two.
\end{Thm}

Recently, Masur-Schleimer \cite{masur-schleimer} found an algorithm to calculate the Hempel distance of a Heegaard splitting with a bounded error term.
The author imagine that their algorithm may also be appliable to bridge decompositions.
However, the point of our result is its practicality: for any given bridge decomposition, we can easily obtain a bridge diagram and check whether it satisfies the well-mixed condition.

\section{Bridge decompositions and the Hempel distance}

Suppose $L$ is a link in $S^3$ and $P$ is a $2$-sphere dividing $S^3$ into two $3$-balls $B_+$ and $B_-$.
Assume that $L$ intersects $P$ transversally and let $\tau _\varepsilon $ be the intersection of $L$ with $B_\varepsilon $ for each $\varepsilon =\pm $.
That is to say, $(S^3,L)$ is decomposed into $T_+:=(B_+,\tau _+)$ and $T_-:=(B_-,\tau _-)$ by $P$.
We call the triple $(T_+,T_-;P)$ an {\it $n$-bridge decomposition} of $L$ if each $T_\varepsilon $ is an $n$-string trivial tangle.
Here, $T_\varepsilon $ is called an {\it $n$-string trivial tangle} if $\tau _\varepsilon $ consists of $n$ arcs parallel to the boundary of $B_\varepsilon $.
Obviously $1$-bridge decompositions are possible only for the trivial knot, so we assume $n\geq 2$ in this paper.

Consider a properly embedded disk $D$ in $B_\varepsilon $.
We call $D$ an {\it essential disk} of $T_\varepsilon $ if $\partial D$ is essential in the surface $\partial B_\varepsilon \setminus \tau _\varepsilon $ and $D$ is disjoint from $\tau _\varepsilon $.
Here, a simple closed curve on a surface is said to be {\it essential} if it neither bounds a disk nor is peripheral in the surface.
Note that essential disks of $T_+$ and $T_-$ are bounded by some essential simple closed curves on the $2n$-punctured sphere $P\setminus L$.

The essential simple closed curves on $P\setminus L$ form a $1$-complex ${\mathcal C}(P\setminus L)$, called the {\it curve graph} of $P\setminus L$.
The vertices of ${\mathcal C}(P\setminus L)$ are the isotopy classes of essential simple closed curves on $P\setminus L$ and a pair of vertices spans an edge of ${\mathcal C}(P\setminus L)$ if the corresponding isotopy classes can be realized as disjoint curves.
In the case of $n=2$, this definition makes the curve graph a discrete set of points and so a slightly different definition is used.

The {\it Hempel distance} (or just the {\it distance}) of $(T_+,T_-;P)$ is defined by
$$d(T_+,T_-):={\rm min}\{ d([\partial D_+],[\partial D_-])\mid D_\varepsilon \text{ is an essential disk of }T_\varepsilon .\ (\varepsilon =\pm )\} $$
where $d([\partial D_+],[\partial D_-])$ is the minimal distance between $[\partial D_+]$ and $[\partial D_-]$ measured in ${\mathcal C}(P\setminus L)$ with the path metric.
Because the curve graph is connected \cite{masur-minsky1}, the distance $d(T_+,T_-)$ is a finite non-negative integer.

For $2$-bridge decompositions, there is a unique essential disk for each of the $2$-string trivial tangles.
Moreover, the curve graph of a $4$-punctured sphere is well understood (see Sections 1.5 and 2.1 in \cite{masur-minsky2} for example) and so we can calculate the exact distance.

Suppose $(T_+,T_-;P)$ is an $n$-bridge decomposition of a link $L$ for $n\geq 3$.
If $d(T_+,T_-)=0$, there are essential disks $D_+,D_-$ of $T_+,T_-$, respectively, such that $[\partial D_+]=[\partial D_-]$.
We can assume $\partial D_+=\partial D_-$ indeed and so $D_+\cup D_-$ is a $2$-sphere in $S^3$.
Therefore, $(T_+,T_-;P)$ is separated by the sphere into an $m$-bridge decomposition and an $(n-m)$-bridge decomposition of sublinks of $L$.
By the definition of essential disks, $m$ is more than $0$ and less than $n$.
Conversely, we can conclude that the distance is at least one if $(T_+,T_-;P)$ is not a such one.

\section{Bridge diagrams and the well-mixed condition}\label{diagram}

Suppose $(T_+,T_-;P)$ is an $n$-bridge decomposition of a link $L$ in $S^3$ and $T_+=(B_+,\tau _+),T_-=(B_-,\tau _-)$.
For each $\varepsilon =\pm $, the $n$ arcs of $\tau _\varepsilon $ can be disjointly projected into $P$.
Let $p:L\rightarrow P$ be such a projection.
A {\it bridge diagram} of $(T_+,T_-;P)$ is a diagram of $L$ obtained from $p(\tau _+)$ and $p(\tau _-)$.
In the terminology of \cite{crowell-fox}, $\tau _+,\tau _-$ are the overpasses and the underpasses of $L$.

Note that the boundary of a regular neighborhood of each arc of $p(\tau _\varepsilon )$ in $P$ bounds an essential disk of $T_\varepsilon $ separating an arc of $\tau _\varepsilon $.
In this sense a bridge diagram represents a family of essential disks of $T_+,T_-$.
So we can think of it as something like a Heegaard diagram for a Heegaard splitting.

It is well known that a bridge decomposition is displayed as a ``plat" as in Figure \ref{fig_braid} (See \cite{birman}).
Now we describe how to convert a plat presentation to a bridge diagram.
For example, consider a $3$-bridge decomposition with a plat presentation as in the left of Figure \ref{fig_example}.
Here $P$ can be isotoped onto any height, so start with $P$ in the position $P_s$.
The top in the right of Figure \ref{fig_example} illustrates a view of a canonical projection of the arcs $t_+^1,t_+^2,t_+^3$ on $P$ from $B_+$ side.
In our pictures, $p(t_+^1),p(t_+^2),p(t_+^3)$ are represented by a solid line, a dotted line, a broken line, respectively.
Shifting $P$ to the position $P_1$, the projections are as the second in the right of Figure \ref{fig_example}.
Shifting $P$ further to the position $P_2$, the projections are as the third.
By continuing this process, the projections are as in Figure \ref{fig_diagram_5} when $P$ is in the position $P_g$.
Then we can find a canonical projection of the arcs $t_-^1,t_-^2,t_-^3$ and obtain a bridge diagram.
\begin{figure}[ht]
\vspace{10pt}
\includegraphics[width=120pt]{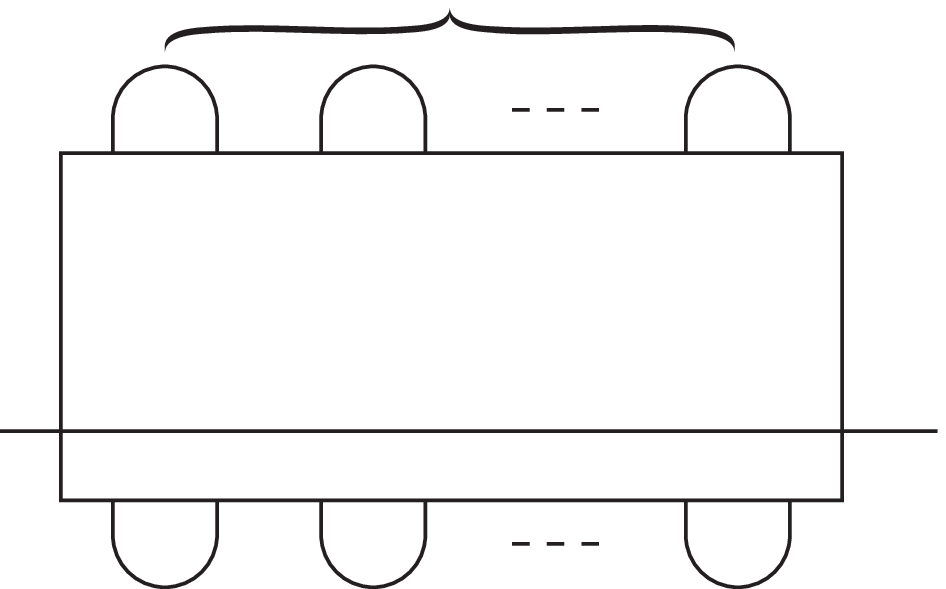}\\
\vspace{-91pt}
\hspace{-7pt}$n$\\[37pt]
\hspace{45pt}$2n$-braid\hspace{35pt}$B_+$\\[1pt]
\hspace{133pt}$P$\\[1pt]
\hspace{120pt}$B_-$
\caption{}
\label{fig_braid}
\end{figure}
\begin{figure}[ht]
\begin{minipage}{160pt}
\begin{center}
\includegraphics[width=120pt]{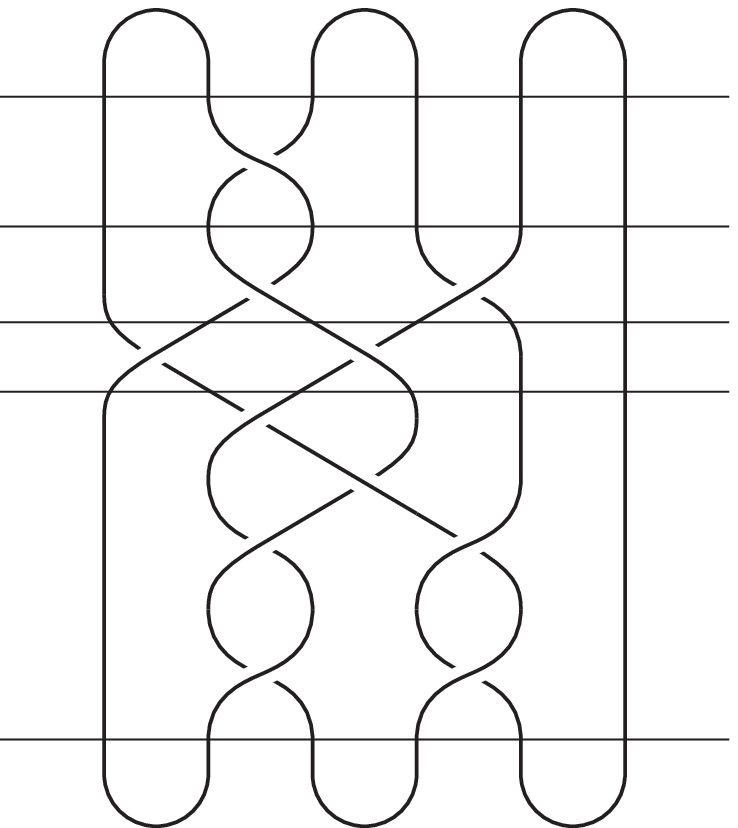}
\end{center}
\vspace{-155pt}
\hspace{41pt}$t_+^1$\hspace{23pt}$t_+^2$\hspace{23pt}$t_+^3$\\[12pt]
\hspace*{140pt}$P_s$\\[9pt]
\hspace*{140pt}$P_1$\\[4pt]
\hspace*{140pt}$P_2$\\
\hspace*{140pt}$P_3$\\
\hspace*{143pt}$\vdots$\\[27pt]
\hspace*{140pt}$P_g$\\[12pt]
\hspace*{41pt}$t_-^1$\hspace{23pt}$t_-^2$\hspace{23pt}$t_-^3$\\
\end{minipage}
\begin{picture}(20,100)(0,0)
\put(-5,70){\line(4,1){30}}
\put(-5,45){\line(1,0){30}}
\put(-5,28){\line(2,-1){30}}
\put(-5,15){\line(2,-3){30}}
\end{picture}
\begin{minipage}{120pt}
\vspace{-38pt}
\hspace{16pt}$p(t_+^1)$\hspace{10pt}$p(t_+^2)$\hspace{10pt}$p(t_+^3)$\\[-16pt]
\begin{center}
\includegraphics[width=100pt]{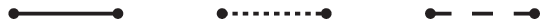}\\[13pt]
\includegraphics[width=100pt]{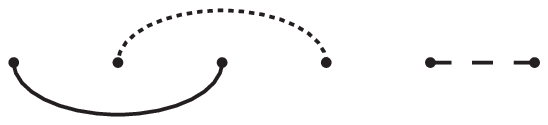}\\[13pt]
\includegraphics[width=100pt]{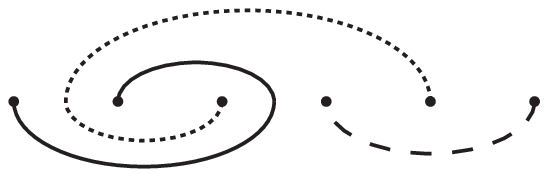}\\[13pt]
\includegraphics[width=100pt]{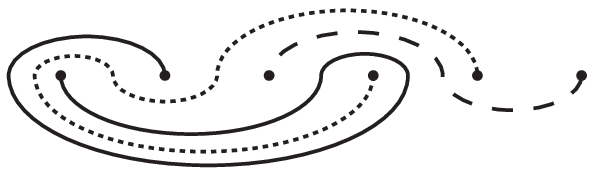}
\end{center}
\end{minipage}
\vspace{-15pt}
\caption{}
\label{fig_example}
\end{figure}
\begin{figure}[ht]
\includegraphics[width=280pt]{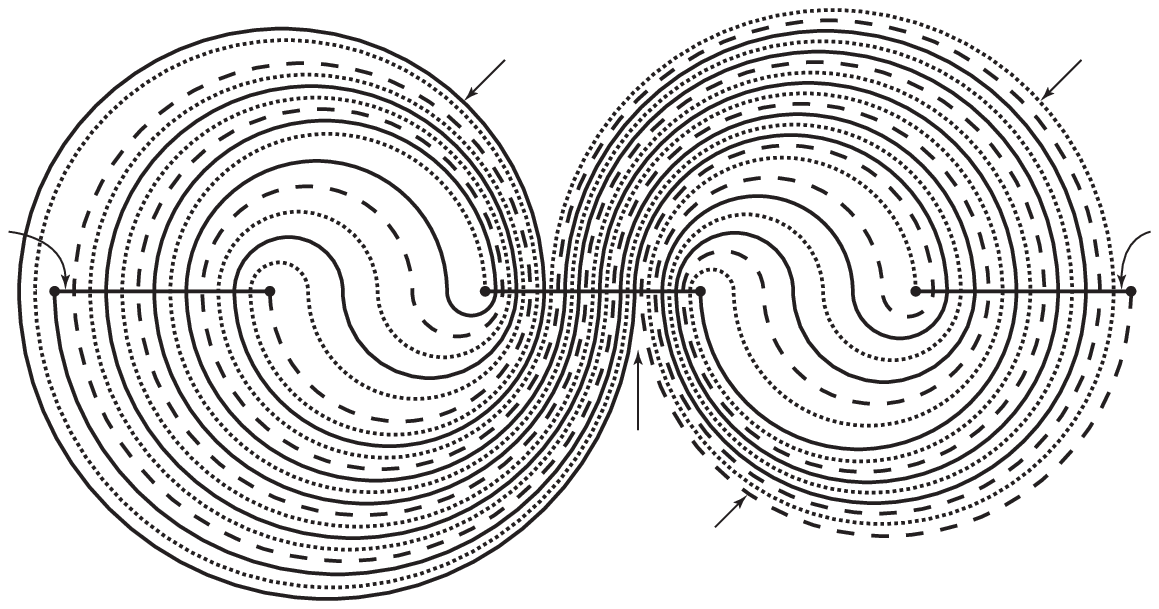}\\
\vspace{-145pt}
\hspace{114pt}$p(t_+^1)$\hspace{110pt}$p(t_+^2)$\\[34pt]
$p(t_-^1)$\hspace{268pt}$p(t_-^3)$\\[41pt]
\hspace{28pt}$p(t_-^2)$\\[9pt]
\hspace{48pt}$p(t_+^3)$
\caption{}
\label{fig_diagram_5}
\end{figure}

Next we study the distance of this $3$-bridge decomposition.
Since the link $L$ is connected, the bridge decomposition cannot be separated into smaller ones.
It follows that the distance is at least one.
Consider the simple closed curve $c$ as in Figure \ref{fig_curve}.
The curve $c$ is essential in $P\setminus L$ and disjoint from both $p(t_+^1)$ and $p(t_-^1)$.
Recall that the boundary of a small neighborhood of $p(t_+^1),p(t_-^1)$ in $P$ bounds an essential disk $D_+^1$ of $T_+$ and an essential disk $D_-^1$ of $T_-$, respectively.
So there are an edge between $[\partial D_+^1],[c]$ and an edge between $[c],[\partial D_-^1]$ in the curve graph ${\mathcal C}(P\setminus L)$.
By definition, the distance is at most two.
It is true that there is no direct edge between $[\partial D_+^1]$ and $[\partial D_-^1]$.
However, this is not enough to conclude that the distance is equal to two because there are infinitely many essential disks of $T_+,T_-$ other than $D_+^1,D_-^1$.
\begin{figure}[ht]
\includegraphics[width=280pt]{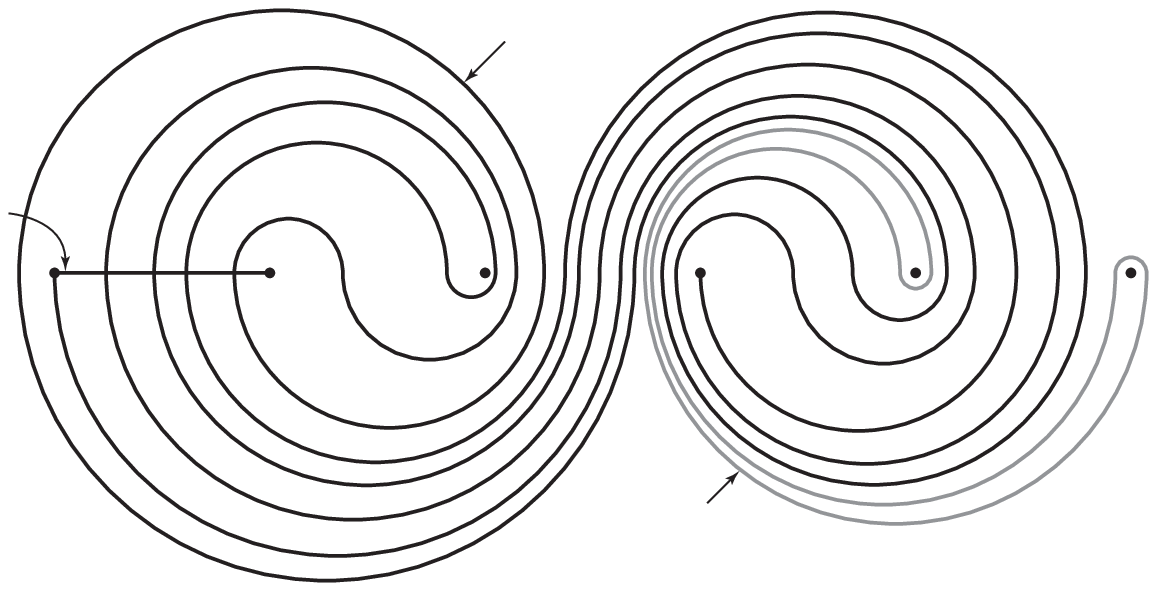}\\
\vspace{-147pt}
\hspace{-18pt}$p(t_+^1)$\\[34pt]
\hspace{-292pt}$p(t_-^1)$\\[58pt]
\hspace{48pt}$c$
\caption{}
\label{fig_curve}
\end{figure}

As shown in \cite{berge}, \cite{casson-gordon}, \cite{lee} and \cite{lustig-moriah}, sufficiently complicated Heegarrd diagram implies a large distance of the Heegaard splitting.
We can expect that sufficiently complicated bridge diagram also implies a large distance of the bridge decomposition.
A bridge diagram should be pretty complicated if it satisfies the {\it well-mixed condition}, which we define in the following.

Denote the arcs of each $\tau _\varepsilon $ by $t_\varepsilon ^1,t_\varepsilon ^2,\ldots ,t_\varepsilon ^n$.
Let $l$ be a loop on $P$ containing $p(\tau _-)$ such that $p(t_-^1),p(t_-^2),\ldots ,p(t_-^n)$ are located in $l$ in this order.
We can assume that $p(\tau _+)$ has been isotoped in $P\setminus L$ to have minimal intersection with $l$.
For the bridge diagram of Figure \ref{fig_diagram_5}, it is natural to choose $l$ to be the closure in $P\cong S^2$ of the horizontal line containing $p(t_-^1)\cup p(t_-^2)\cup p(t_-^3)$.
Let $H_+,H_-\subset P$ be the hemi-spheres divided by $l$ and let $\delta _i$ ($1\leq i\leq n$) be the component of $l\setminus p(\tau _-)$ which lies between $p(t_-^i)$ and $p(t_-^{i+1})$.
(Here the indices are considered modulo $n$.)
Let ${\mathcal A}_{i,j,\varepsilon }$ be the set of components of $p(\tau _+)\cap H_\varepsilon $ separating $\delta _i$ from $\delta _j$ in $H_\varepsilon $ for a distinct pair $i,j\in \{ 1,2,\ldots ,n\} $ and $\varepsilon \in \{ +,-\} $.
For example, Figure \ref{fig_diagram_6} displays ${\mathcal A}_{1,2,+}$ for the above bridge diagram.
Note that ${\mathcal A}_{i,j,\varepsilon }$ consists of parallel arcs in $H_\varepsilon $.
\begin{figure}[ht]
\includegraphics[width=280pt]{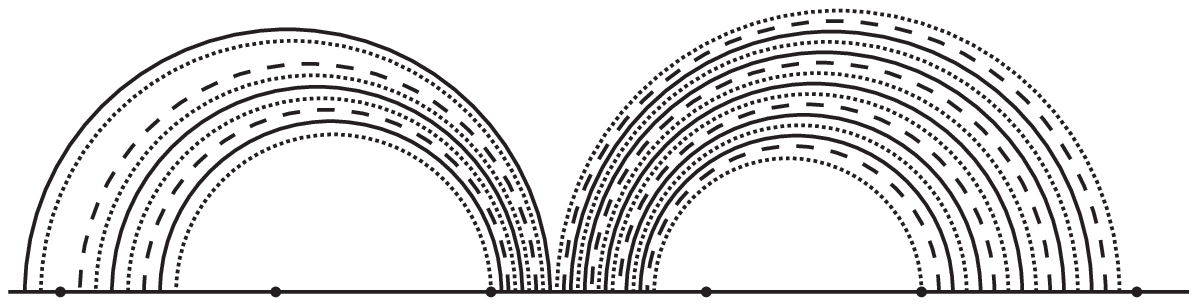}\\
\vspace{-50pt}
\hspace{270pt}$H_+$\\[17pt]
\hspace{286pt}$l$\\[-4pt]
\hspace{26pt}$p(t_-^1)$\hspace{35pt}$\delta _1$\hspace{33pt}$p(t_-^2)$\hspace{35pt}$\delta _2$\hspace{32pt}$p(t_-^3)$\hspace{17pt}$\delta _3$
\caption{}
\label{fig_diagram_6}
\end{figure}

\begin{Def}
\begin{enumerate}
\item A bridge diagram satisfies the {\it $(i,j,\varepsilon )$-well-mixed condition} if in ${\mathcal A}_{i,j,\varepsilon }\subset H_\varepsilon $, a subarc of $p(t_+^r)$ is adjacent to a subarc of $p(t_+^s)$ for all distinct pair $r,s\in \{ 1,2,\ldots ,n\} $.
\item A bridge diagram satisfies the {\it well-mixed condition} if it satisfies the $(i,j,\varepsilon )$-well-mixed condition for all combinations of a distinct pair $i,j,\in \{ 1,2,\ldots ,n\} $ and $\varepsilon \in \{ +,-\} $.
\end{enumerate}
\end{Def}

As in Figure \ref{fig_diagram_6}, the bridge diagram in Figure \ref{fig_diagram_5} amply satisfies the $(1,2,+)$-well-mixed condition.
One can also check the $(i,j,\varepsilon )$-well-mixed condition for all the other combinations $(i,j,\varepsilon )=(1,2,-),(2,3,+),(2,3,-),(3,1,+),(3,1,-)$.
Hence the bridge diagram in Figure \ref{fig_diagram_5} satisfies the well-mixed condition.

\section{Proof of the theorem}

Firstly, consider an essential disk $D_-$ of $T_-$.
Assume that $D_-$ has been isotoped so that $|\partial D_-\cap l|$ is minimal.
Here, $|\cdot |$ denotes the number of connected components of a topological space.

\begin{Lem}
There exist a distinct pair $i,j\in \{ 1,2,\ldots ,n\} $ and $\varepsilon \in \{ +,-\} $ such that $\partial D_-$ includes a subarc connecting $\delta _i$ and $\delta _j$ in $H_\varepsilon $.
\end{Lem}

\begin{proof}
Since the arcs of $\tau _-$ are projected to subarcs of $l$, there exists a disk $E_-$ in $B_-$ such that $\partial E_-=l$ and $\tau _-\subset E_-$.
The essential disk $D_-$ must have non-empty intersection with $E_-$.
The closed components of $D_-\cap E_-$ can be eliminated by an isotopy of ${\rm Int}D_-$.
Then $D_-\cap E_-$ is a non-empty family of properly embedded arcs in $D_-$.
Consider an outermost subdisk $D_-^0$ of $D_-$ cut off by an arc of them.
For the minimality of $|\partial D_-\cap l|$, we can see that $\partial D_-^0\cap \partial D_-$ connects $\delta _i$ and $\delta _j$ in $H_\varepsilon $ for a distinct pair $i,j\in \{ 1,2,\ldots ,n\} $ and $\varepsilon \in \{ +,-\} $.
\end{proof}

Secondly, consider an essential disk $D_+$ of $T_+$.
Assume that $D_+$ has been isotoped so that $|\partial D_+\cap p(\tau _+)|$ is minimal.

\begin{Lem}
Suppose $c$ is an essential simple closed curve on $P\setminus L$ disjoint from $\partial D_+$.
There exist a distinct pair $r,s\in \{ 1,2,\ldots ,n\} $ such that no subarc of $c$ connects $p(t_+^r)$ and $p(t_+^s)$ directly (i.e. its interior is disjoint from $p(\tau _+)$).
\end{Lem}

\begin{proof}
Let $E_+^i$ be a disk of parallelism between $t_+^i$ and $p(t_+^i)$ for each $i=1,2,\ldots ,n$ so that $E_+^1,E_+^2,\ldots ,E_+^n$ are pairwise disjoint.
The closed components of $D_+\cap (E_+^1\cup E_+^2\cup \cdots \cup E_+^n)$ can be eliminated by an isotopy of ${\rm Int}D_+$.
If $D_+\cap (E_+^1\cup E_+^2\cup \cdots \cup E_+^n)$ is empty, $D_+$ separates the $n$ disks $E_+^1,E_+^2,\ldots ,E_+^n$ into two classes in $B_+$.
Since $D_+$ is essential, both these classes are not empty.
If $D_+\cap (E_+^1\cup E_+^2\cup \cdots \cup E_+^n)$ is not empty, it consists of properly embedded arcs in $D_+$.
Consider an outermost subdisk $D_+^0$ of $D_+$ cut off by an arc of them, say, an arc of $D_+\cap E_+^k$.
Then, $D_+^0\cup E_+^k$ separates the $(n-1)$ disks $E_+^1,\ldots ,E_+^{k-1},E_+^{k+1},\ldots ,E_+^n$ into two classes in $B_+$.
Since $|\partial D_+\cap p(t_+^k)|$ is minimal, both these classes are not empty.
Anyway, by choosing $r$ and $s$ from the indexes of the disks of separated classes, the lemma follows.
\end{proof}

Assume that the distance of $(T_+,T_-;P)$ is less than two.
There are disjoint essential disks $D_+,D_-$ of $T_+,T_-$, respectively.
If $\partial D_-$ contains a subarc connecting $\delta _i$ and $\delta _j$ in $H_\varepsilon $, it intersects all the arcs of ${\mathcal A}_{i,j,\varepsilon }$.
In particular, if two arcs of ${\mathcal A}_{i,j,\varepsilon }$ are adjacent in $H_\varepsilon $, a subarc of $\partial D_-$ connects them directly.
The above observations and the well-mixed condition are almost enough to lead to a contradiction, but only the following should be checked:

\begin{Lem}
The disks $D_+$ and $D_-$ can be isotoped preserving the disjointness so that $|\partial D_+\cap p(\tau _+)|$ and $|\partial D_-\cap l|$ are minimal.
\end{Lem}

\begin{proof}
Note that any isotopy of $\partial D_\varepsilon $ in $P\setminus L$ can be realized by an isotopy of $D_\varepsilon $ in $B_\varepsilon \setminus \tau _\varepsilon $ for $\varepsilon =\pm $.

If $|\partial D_+\cap p(\tau _+)|$ is not minimal, there are a subarc of $\partial D_+$ and a subarc $\alpha $ of $p(\tau _+)$ cobounding a disk $\Delta _+$ in $P\setminus L$.
Since $D_+,D_-$ are disjoint, $\partial D_-\cap \Delta _+$ consists of arcs parallel into $\alpha $.
Let $\Delta _+^0$ be an outermost disk of the parallelisms.
By assumption, $p(\tau _+)$ has minimal intersection with $l$ and so no component of $l\cap \Delta _+^0$ has both end points on $\alpha $.
By an isotopy of $\partial D_-$ across $\Delta _+^0$, we can reduce $|\partial D_-\cap \Delta _+|$ without increasing $|\partial D_-\cap l|$.
After pushing out $\partial D_-$ from $\Delta _+$ in this way, we can reduce $|\partial D_+\cap p(\tau _+)|$ by an isotopy of $\partial D_+$ across $\Delta _+$.

If $|\partial D_-\cap l|$ is not minimal, there are a subarc of $\partial D_-$ and a subarc $\beta $ of $l$ cobounding a disk $\Delta _-$ in $P\setminus L$.
The intersection $\partial D_+\cap \Delta _-$ consists of arcs parallel into $\beta $.
Let $\Delta _-^0$ be an outermost disk of the parallelisms.
By the minimality of $|l\cap p(\tau _+)|$, no component of $p(\tau _+)\cap \Delta _-^0$ has both end points at $\beta $.
By an isotopy of $\partial D_+$ across $\Delta _-^0$, we can reduce $|\partial D_+\cap \Delta _-|$ without increasing $|\partial D_+\cap p(\tau _+)|$.
After pushing out $\partial D_+$ from $\Delta _-$ in this way, we can reduce $|\partial D_-\cap l|$ by an isotopy of $\partial D_-$ across $\Delta _-$.
\end{proof}

Theorem \ref{main} implies that the $3$-bridge decomposition in Figure \ref{fig_example} has distance at least two.
Since we have shown that it is at most two, the distance is exactly two.
We can work out in this way fairly many $n$-bridge decompositions, especially for $n=3$.

\subsection*{Acknowledgement}

I would like to thank Jang~Yeonhee for giving me the main question of this work and helpful conversations.
I would like to thank\break Ken'ichi~Ohshika for all his help as a mentor.
I would also like to thank\break  Makoto~Ozawa and Makoto~Sakuma for valuable comments and suggestions.


\begin{thebibliography}{99}
\bibitem{bachman-schleimer}D. Bachman and S. Schleimer, {\it Distance and bridge position}, Pacific J. Math. {\bf 219} (2005), no. 2, 221--235.
\bibitem{berge}J. Berge, {\it A closed orientable $3$-manifold with distinct distance three genus two Heegaard splittings}, arXiv:math.GT/0912.1315.
\bibitem{birman}J. S. Birman, {\it Plat presentations for link groups}, Collection of articles dedicated to Wilhelm Magnus. Comm. Pure Appl. Math. {\bf 26} (1973), 673--678.
\bibitem{casson-gordon}A. Casson and C. Gordon, {\it Manifolds with irreducible Heegaard splittings of arbitrary large genus}, Unpublished.
\bibitem{crowell-fox}R. Crowell and R. Fox, Introduction to knot theory, Graduate Texts in Mathematics, {\bf 57}, Springer-Verlag, New York-Heidelberg.
\bibitem{hartshorn}K. Hartshorn, {\it Heegaard splittings of Haken manifolds have bounded distance}, Pacific J. Math. {\bf 204} (2002), no. 1, 61--75.
\bibitem{hempel}J. Hempel, {\it 3-manifolds as viewed from the curve complex}, Topology {\bf 40} (2001), no. 3, 631--657.
\bibitem{lee}J. H. Lee, {\it Rectangle condition for irreducibility of Heegaard splittings}, arXiv:math.GT/0812.0225.
\bibitem{lustig-moriah}M. Lustig and Y. Moriah, {\it High distance Heegaard splittings via fat train tracks}, Topology Appl. {\bf 156} (2009), no. 6, 1118--1129.
\bibitem{masur-minsky1}H.A. Masur and Y.N. Minsky, {\it Geometry of the complex of curves. I. Hyperbolicity}, Invent. Math. {\bf 138} (1999), no. 1, 103--149.
\bibitem{masur-minsky2}H.A. Masur and Y.N. Minsky, {\it Geometry of the complex of curves. II. Hierarchical structure}, Geom. Funct. Anal. {\bf 10} (2000), no. 4, 902--974.
\bibitem{masur-schleimer}H. Masur and S. Schleimer, {\it The geometry of the disk complex}, arXiv:math.GT/1010.3174.
\bibitem{scharlemann-tomova}M. Scharlemann and M. Tomova, {\it Alternate Heegaard genus bounds distance}, Geom. Topol. {\bf 10} (2006), 593--617.
\bibitem{tomova}M. Tomova, {\it Multiple bridge surfaces restrict knot distance}, Algebr. Geom. Topol. {\bf 7} (2007), 957--1006.
\end{thebibliography}
\end{document}